\documentclass[11pt,twoside, a4paper, english, reqno]{amsart}
\usepackage{amssymb,amsfonts,latexsym,color}
\usepackage{graphics,verbatim}
\usepackage{graphicx}
\usepackage{a4wide}
\usepackage{hyperref}
\newtheorem{theorem}{Theorem}[section]

\newtheorem{lemma}{Lemma}[section]

\newtheorem{definition}{Definition}[section]

 \newcommand{\<}{\left\langle}
\renewcommand{\>}{\right\rangle}

\newcommand{\eps}{\varepsilon}

\newcommand{\Real}{\mathbb{R}}

\newcommand{\norm}[1]{\left\Vert#1\right\Vert}

\newcommand{\be} {\begin{equation}}
\newcommand{\ee} {\end{equation}}
\newcommand{\bea} {\begin{eqnarray}}
\newcommand{\eea} {\end{eqnarray}}
\newcommand{\Bea} {\begin{eqnarray*}}
\newcommand{\Eea} {\end{eqnarray*}}
\newcommand{\pa} {\partial}

\newcommand{\al} {\alpha}
\newcommand{\ba} {\beta}
\newcommand{\de} {\delta}
\newcommand{\na}{\nabla}
\newcommand{\ga} {\gamma}
\newcommand{\Ga} {\Gamma}
\newcommand{\Om} {\Omega}

\newcommand{\De} {\Delta}
\newcommand{\la} {\lambda}

\newcommand{\no} {\nonumber}

\newcommand{\f}{\frac}

\newcommand{\R}{\mathbb R}
\newcommand{\N}{\mathbb N}
\newcommand{\Rn}{\mathbb R^N}
\newcommand{\Iom}{\int_{\Omega}}
\newcommand{\deb}{\rightharpoonup}
\catcode`\@=11

\makeatletter \@addtoreset{equation}{section} \makeatother

\begin{document}
	\title[On Nonlocal elliptic equations with jumping nonlinearity]{Multiplicity results for non-local elliptic problems with jumping nonlinearity}
	
	\author{Debangana Mukherjee}
	\address{Debangana Mukherjee, Department of Mathematics and Information Technology, Montanuniversit\"at, Leoben, Franz-Josef-Strasse 18, 8700 Leoben, Austria}
	\email{debangana18@gmail.com}

	\subjclass[2010]{Primary 35J20, 35J25,  35J60, 35R11 }
	\keywords{Fractional Laplacian, Jumping Nonlinearities, Multiple Solutions, Sign-Changing Solutions, Non-local operator, Fu\v cik Spectrum.}
	\maketitle
	\date{}
	
	\begin{abstract} The present paper studies the fractional $p$-Laplacian boundary value problems with jumping nonlinearities at zero or infinity and obtain the existence of multiple solutions and sign-changing solutions by constructing the suitable pseudo-gradient vector field of the corresponding energy functional.
	\end{abstract}
	\tableofcontents
 \section{ Introduction}
 In this article, we  consider the following problem:
\begin{equation*}
	(\mathcal{P})
	\left\{\begin{aligned}
		(-\De)^s_p u  &=h(u)\quad\text{in }\quad \Om, \\
		u &= 0  \quad\text{in }\quad \Rn \setminus \Om,
	\end{aligned}
	\right.
\end{equation*}
where $\Om \subset \Rn$ is open, smooth, bounded domain with smooth boundary, $s \in (0,1), p>1,N<ps$;
the non-local Operator $(-\Delta)^s_p$ is defined as follows:
\begin{align} \label{frac s_p}
	(-\Delta)^s_p u(x)=\lim_{\eps\to 0}\int_{\mathbb{R}^N\setminus B_\eps(x)}\frac{|u(y)-u(x)|^{p-2}(u(y)-u(x))}{|x-y|^{N+ps}}dy,\,\,\,x\in\mathbb{R}^N,
\end{align}
 and we assume $h(u)$ has `jumping' nonlinearities at zero or infinity:
\begin{itemize}
	\item [($A_1$)]
	$\lim_{u \to 0^+} \frac{h(u)}{|u|^{p-2}u}=a_0 ;\,\  \lim_{u \to 0^-}\frac{h(u)}{|u|^{p-2}u}=d_0,$
	\item [($A_2$)]
	$\lim_{u \to +\infty} \frac{h(u)}{|u|^{p-2}u}=a_1 ;\,\   \lim_{u \to -\infty}\frac{h(u)}{|u|^{p-2}u}=d_1,$
	\item [($A_3$)]
	$h(0)=0,\,\ h(u)u \geq 0$ and $h(u)$ is locally Lipschitz.
\end{itemize}
In the classical (local) case for $p=2,$ this problem has been extensively  studied by several authors \cite{Bartsch}, \cite{Chang-1}, \cite{Chang-2}, \cite{Dancer-2}, \cite{Dancer-1}, \cite{Hofer}, \cite{Li-1}, \cite{Li-2}, \cite{Li-Zhang} by using Morse theory, linking methods, the technique of descent flow of the gradient 
and we cite \cite{Perera-Carl},\cite{Dancer-2},\cite{Dancer-1},\cite{Li-1},\cite{Li-2} for the jumping nonlinearities.

A classical topic in nonlinear analysis is the study of existence and multiplicity of solutions for nonlinear equations. In the past, significant amount of research is carried out for studying the following general boundary value problem
\begin{equation*}
	\left\{\begin{aligned}
		-\De u  &=h(u)\quad\text{in }\quad \Om, \\
		u &= 0  \quad\text{on }\quad  \pa \Om.
	\end{aligned}
	\right.
\end{equation*}

In \cite{Dancer-2}, the authors have considered  the Dirichlet boundary value problems with jumping nonlinearities at zero or infinity and obtained multiplicity of solutions and their results heavily depends on the property of $h(u).$ More precisely, they have considered 
\begin{equation*}
	h(u) =
	\left\{\begin{aligned}
	  &au-\al u^2,\,\  u\geq 0,\\
		&du+\ba u^2,\,\ u\leq 0,
	\end{aligned}
	\right.
\end{equation*}  
where $\al \geq 0, \ba>0$ and $a,d>\la_1,$ $\la_1$ being the first eigenvalue of Dirichlet Laplacian. Using variational methods, thay have established existence of sign-changing solutions.  In \cite{Li-1}, the authors have obtained up to six non-trivial solutions, out of which two are sign-changing, two are positive and two negative. In \cite{Li-3}, the authors have studied asymptotically linear and superlinear boundary value problems involving a reaction nonzero at zero and obtained multiple solutions and sign-changing solutions (up to five non-trivial solutions) and the sign-changing solutions change sign exactly once.
A substancial amount of work has been followed through
to obtain the sign-changing solutions and multiple 
solutions for elliptic boundary value problems which has been accomplished by many authors,
see \cite{Amann}, \cite{Bartsch}, \cite{Guo-Sun}, \cite{Hofer} and the references therein.

 In the local case, for $p \neq 2,$ existence of solutions and multiple solutions have been studied in \cite{Costa}, \cite{Figueiredo}, \cite{Dancer-3}, \cite{ Drabek}, \cite{Drabek-1}, \cite{Perera-1}, \cite{Su-Liu} but in these papers, no work on sign-changing solution have been considered.  The existence of sign-changing solutions for $p \neq 2$ has been studied in \cite{Zhang-Li}.  In the non-local case, multiplicity results and existence of sign-changing solutions have been discussed by several authors, we refer a few among them (see \cite{Bhakta-3}, \cite{Bhakta-1}, \cite{Bhakta-2}, \cite{Molica} and the refernces therein). Very recently, in \cite{Bhakta-3}, the authors have established the existence of infinitely many nontrivial solutions for the class of $(p,q)$ fractional elliptic equations in bounded domains in $\Rn.$


In the classical (local) case for the $p$-Laplace operator, this problem is addressed in \cite{Zhang-Li}.
In the spirit of \cite{Zhang-Li}, we study the existence of multiple solutions including sign-changing solutions for the problem $(\mathcal{P})$. 
In general, the existence of solutions is obtained by Mountain Pass Lemma. 
For the existence of sign-changing solutions, the authors in 
~\cite{Zhang-Li} have constructed the pseudo-gradient vector field in $W^{1,p}_0(\Om)$ and then using the dynamics theory in Banach spaces, they have proved sign-changing and multiple solutions. We have incorporated the same technique but in the non-local setting. Unlike the local case, such as the value $\De u(x)$, is computed using the value of $u$ in an arbitrarily small neighborhood of $x$, non-local operators need information about the values of $u$ throughout the whole domain. The rectitude of our work lies in overcoming the difficulties arised due to fractional scheme.
As far as we know, such result for existence of multiple and sign-changing solutions for all $p$ lying in the non-local framework, is not available in the literature.

%

\subsection{Functional Setting}
We denote the standard fractional Sobolev space by $W^{s,p}(\Omega)$ endowed with the norm
$$
\|{u}\|_{W^{s,p}(\Om)}:=\|{u}\|_{L^p(\Om)}+\left(\int_{\Om\times\Om} \frac{|u(x)-u(y)|^p}{|x-y|^{N+sp}}dxdy\right)^{1/p},
$$  for $p\geq 1$ and $s\in(0,1).$
We define 
$$
X:=\Big\{u:\mathbb{R}^N\to\mathbb{R}\mbox{ measurable }\Big|u|_{\Omega}\in L^p(\Omega)\mbox{ and }
\int_{Q} \frac{|u(x)-u(y)|^p}{|x-y|^{N+sp}}dxdy<\infty\Big\}.
$$
where $Q=\R^{2N}\setminus (C\Om \times C\Om)$ with $C\Om=\Rn \setminus \Om. $
The space $X$ is endowed with the following norm,
$$||u||_p=\big(\Iom |u|^pdx\big)^{\frac{1}{p}}+\left(\int_{Q} \frac{|u(x)-u(y)|^p}{|x-y|^{N+sp}}dxdy\right)^{1/p}.$$

We define $X_{0}(\Om) :=\Big\{u \in X:u=0 \quad\text{a.e. in}\quad \Rn \setminus \Om\Big\}$ and  
for any $p>1$, $X_{0}(\Om)$ is a uniformly convex Banach space endowed with the norm   
$$||u||_{X_0(\Om)}=\left(\int_{Q} \frac{|u(x)-u(y)|^p}{|x-y|^{N+sp}}dxdy\right)^{1/p}.$$

Since $u=0$ in $\Rn\setminus\Om,$ the above integral can be extended to all of $\mathbb{R}^N.$ As $\Om\subset \mathbb{R}^n$ be smooth bounded domain, therefore, for $N<sp$, the embedding $X_0(\Om) \hookrightarrow C^{0,\alpha}(\Om)$ is continuous with $\al:=\frac{sp-N}{p}$, (see Theorem 8.2, \cite{ NePaVal}).

\subsection{Dancer-{F}u\v cik spectrum }
Let $\Om \subset \Rn$ be a bounded  Lipschitz domain. 
Let us consider the following equation:
\begin{equation}\label{P-1}
\left\{\begin{aligned}
(-\De)^s_p u  &=a(u^+)^{p-1}-d(u^-)^{p-1} \quad\text{in }\quad \Om, \\
u &= 0  \quad\text{in }\quad \Rn \setminus \Om.
\end{aligned}
\right.
\end{equation} 
where $a \in \R, d \in \R, s \in (0,1), p>1$ and $u^+=\max \{u,0\},\,\ u^-=-\min \{u,0\}.$ 

\begin{definition}(weak solution)
 A weak solution of  (\ref{P-1}) is defined as $u \in X_0(\Om)$  such that $u$ satisfies
\begin{align}\label{weak-sol}
	&\int_{\R^{2N}}\frac{|u(x)-u(y)|^{p-2}(u(x)-u(y))(v(x)-v(y))}{|x-y|^{N+sp}}dxdy\no\\
	&\qquad=	\Iom [a(u^+)^{p-1}-d(u^-)^{p-1}]vdx \,\ \text{for all}\,\ v \in X_0(\Om).
\end{align}	
\end{definition}

\begin{definition}(Dancer-{F}u\v cik spectrum)
The Dancer-Fu\v cik spectrum of fractional $p$-Laplacian operator $(-\De)^s_p$ in $\Om$ is denoted by $\Sigma_p^s(\Om)$ and is defined as:
$$\Sigma_p^s(\Om)=\big\{ (a,d) \in \R^2: (\ref{P-1})\,\ \text{has nontrivial weak solution} \,\ u  \,\ \text{in}\,\ X_0(\Om) \big\}.$$
\end{definition}
If we take $a=d:=\la,$ then (\ref{P-1}) can be written as:
\begin{equation}\label{P--}
\left\{\begin{aligned}
(-\De)^s_p u  &=\la |u|^{p-2}u \quad\text{in }\quad \Om, \\
u &= 0  \quad\text{in }\quad \Rn \setminus \Om.
\end{aligned}
\right.
\end{equation} 
Therefore, the Dancer-Fu\v cik spectrum of (\ref{P-1}) is same as the usual spectrum of (\ref{P--}).
Let us denote the spectrum of (\ref{P--}) as the set:= $\{\la_k\}_{k=1}^{\infty}.$ We note that for all $k \in \N, (\la_k,\la_k) \in \Sigma_p^s(\Om).$ Moreover,  the two lines $\{\la_1 \times \R \}$ and $\R \times \{\la_1 \} \subset \Sigma_p^s(\Om).$ In the case of $s=1$ and $p=2,$ (\ref{P-1}) becomes the Dancer-Fu\v cik spectrum of the Dirichlet Laplacian. This notion of Fu\v cik spectrum was first introduced by Fu\v cik in \cite{Fucik} and Dancer in \cite{Dancer},\cite{Dancer-4}.


For $s=1$ and $p \neq 2,$ Equation (\ref{P-1}) coincides with the Dancer-Fu\v cik spectrum for 
the p-Laplacian operator $\De_pu=\text{div}(|u|^{p-2}\na u)$ and in this case, this problem has been researched by many authors, see for instance, \cite{Drabek-2}, \cite{Perera-1}. 


In recent literature, lot of work is done concerning the Dancer-Fu\v cik spectrum in the non-local framework. In \cite{Sarika-3}, Goyal and Sreenadh have studied the Dancer-Fu\v cik spectrum for the fractional Laplace operator $(-\De)^s.$ In \cite{Sarika-1}, the authors have worked on the Dancer-Fu\v cik spectrum for the $p$-fractional Laplace operator with non-local normal derivative conditions. In \cite{Sarika-2}, Goyal has discussed the Fu\v cik spectrum of $p$-fractional Hardy Sobolev operator with weight function. Recently, in \cite{Brasco-Parini}, the authors have researched about the second eigenvalue for a fractional $p$-Laplace operator and showed that the second variational eigenvalue $\la_2$ is bigger than $\la_1$ and $(\la_1,\la_2)$ does not contain other eigenvalues. The existence of Fu\v cik eigenvalues for fractional $p$-Laplace operator with Dirichlet boundary conditions have been discussed by many authors, see for instance \cite{Perera-3}, \cite{Rossi}.

Let $\la_1:=\inf \{\norm{u}_{X_0(\Om)}^p : u \in X_0(\Om),\, |u|_p=1 \}.$ It is known in \cite{Lindqvist}
that $\la_1$ (the principal eigen value )
is simple and there exists $\phi_1 \in X_0(\Om) \cap C^{\ba}(\Om)$  where $\ba=s-\frac{2N}{p}$ ( if $sp>2N$) such that $|\phi_1|_p=1$ and $\phi_1>0$ in $\Om.$ It is also known \cite{Lindqvist} 
that $\la_1$ is in the isolated spectrum and
$$\la_2:=\min \{\la \in \R: \la \,\ \text{is an eigen-value and}\,\ \la>\la_1 \}.$$
Using the properties of Dancer-Fu\v cik spectrum, it is easy to construct a nontrivial curve in $\Sigma^s_p(\Om).$ With this, we define  $\eta : (\la_1,\la_2] \to \R$ such that $\eta$ has the following properties:
\begin{itemize}
	\item [(a)]
	$\eta$ is  a continous function, graph of $\eta$ lies in the $ad-$plane;
	\item[(b)]
	$\eta(\la_2)=\la_2, \lim_{\la \to \la_1+0}\eta(\la)=+\infty$,$\eta$ is strictly decreasing;
	\item [(c)]
	(\ref{P-1}) has a non-trivial solution for $(a,d)=(a,\eta(a)), a \in (\la_1,\la_2]$ and $(a,d)=(\eta(d),d), d \in (\la_1,\la_2]$;
	\item [(d)]
	(\ref{P-1}) has no nontrivial solution for $\la_1<d<\eta(a), a \in (\la_1,\la_2]$
	or $\la_1<a<\eta(d), d\in (\la_1,\la_2].$
\end{itemize}	
For details, see   [section 8.2 \cite{Perera-2}].


Let us consider the curve $\Ga$ defined by:
\begin{equation}\label{Ga}
\Ga:=\{(a,\eta(a)) : \la_1<a \leq \la_2 \} \cup \{(\eta(d),d):\la_1<d\leq \la_2 \}
\end{equation}
and 
\begin{equation}\label{S}
S:=\{ (a,d) \in \R^2: (a,d)\,\  \text{lies above}\,\ \Ga \}. 
\end{equation}

The main results of this article are the following:
\begin{theorem}\label{Theorem-1}
Let $a_1,d_1 \in (-\infty, \la_1)$, $(a_0,d_0) \in S$  with $h$ satisfying the assumptions $(A_1)$, $(A_2)$ and $(A_3)$. Then  problem $(\mathcal{P})$ has at least three non-trivial solutions, at least one is positive, one is negative and one is sign-changing.
\end{theorem}	

\begin{theorem}\label{Theorem-2}
	Let $(a_1,d_1) \in S \setminus \Sigma_p^s(\Om), a_0,d_0 < \la_1$ and
	$h$ satisfy the assumptions $(A_1)$,$(A_2)$ and $(A_3)$.  Then problem $(\mathcal{P})$ has at least three non-trivial solutions, at least one is positive, one is negative and one is sign-changing.
\end{theorem}	
Let us briefly describe the contents of this paper: Section-\ref{S-1} consists of the preliminaries followed by some notations and definitions. In Section-\ref{S-2}, we meticulously
construct a suitable pseudo-gradient vector field which is a pivotal tool to establish our main results. Section-\ref{S-3} provides the details about the corresponding energy functional satisfying Palais-Smale condition. In Section-\ref{S-4}, we provide a technical lemma dealing with the construction of a curve which connects the interior of positive and negative cones. Finally, in Section-\ref{S-5}, we prove our main results.

\section{Preliminaries}\label{S-1}
\textit{Energy functionals}:
Let us consider the Euler-Lagrange energy functionals corresponding to (\ref{P-1}) is given by: $E_{(a,d)} : X_0(\Om) \to \R$ by
\begin{equation}\label{E-a,d}
E_{(a,d)}(u):=\frac{1}{p}\norm{u}_{X_0(\Om)}^p-\frac{a}{p}|{u^+}|_p^p-\frac{d}{p}|{u^-}|_p^p,\,\ \text{for all}\,\ u \in X_0(\Om),
\end{equation}
and $E:X_0(\Om) \to \R$, corresponding to $(\mathcal{P})$, by
\begin{equation}\label{E}
E(u)=\frac{1}{p}\norm{u}_{X_0(\Om)}^p-\Iom H(u), \,\ \text{for all}\,\ u \in X_0(\Om),
\end{equation}
where $H(t)=\int_{0}^{t}h(s)ds.$ It is well-known that both $E_{(a,d)}, E \in C^1(X_0(\Om),\R)$ with 
\bea\label{E'-a,d}
\<E'_{(a,d)}(u),\phi \>&=&\int_{\R^{2N}}\frac{|u(x)-u(y)|^{p-2}(u(x)-u(y))((\phi(x)-\phi(y))}{|x-y|^{N+ps}}dxdy\no\\
&-&a\Iom (u^+)^{p-1}\phi dx-d\Iom (u^-)^{p-1}\phi dx,
\eea
for all $\phi \in X_0(\Om)$ and 
\bea\label{E'}
\<E'(u),\phi \>&=&\int_{\R^{2N}}\frac{|u(x)-u(y)|^{p-2}(u(x)-u(y))((\phi(x)-\phi(y))}{|x-y|^{N+ps}}dxdy\no\\
&-&\Iom h(u) \phi dx, \,\ \text{for all}\,\ \phi \in X_0(\Om).
\eea
We know that the critical points of $E_{(a,d)}$ and $E$ are solutions of (\ref{P-1}) and $(\mathcal{P})$ respectively.
\subsection{Notations}: We list below some notations we used throughout the paper: 
\begin{itemize}
	\item[-]
	$|u|_p$ denotes the norm in the space $L^p(\Om).$
	\item [-]
	$\Sigma_p^s(\Om)$ denotes the Dancer-Fu\v cik spectrum of $(-\De)^s_p$ in $\Om$.
	\item [-]
	$P^\circ$ denotes Interior of $P$ in $X_0(\Om).$
	\item [-]
	for $t \in \R,$ $t^+:=\max\{0,t\}$ and $t^-:=\max\{0,-t\}.$
	
\end{itemize}	


\subsection{Some definitions}:
We recall  some definitions which are needed to establish our results. 

\begin{definition}
	Let $X$ be a Banach space, $\phi \in C^1(X,\R)$ and $Y=\{u \in X: \phi'(u) \neq 0  \}.$ A pseudo-gradient vector field for $\phi$ on $Y$ is a locally Lipschitz continuous mapping $V:Y \to X$ such that there exists $\al,\ba >0$ such that for every $u \in Y,$ one has
	$$\norm{V(u)} \leq \al \norm{\phi'(u)},\,\ \<\phi'(u),V(u)\> \geq \ba \norm{\phi'(u)}^2.$$
\end{definition}

	Let $\mathcal{M}=\{u \in X_0(\Om): E'(u)=0\}.$ Let $u(t,u_0),$ for $0 \leq t< \eta(u_0),$ be the solution of the initial value problem
 \begin{equation}\label{V-u}
 	\left\{\begin{aligned}
\frac{du}{dt}(t)&=-V(u(t)),\quad t\in [0,\eta(u_0)), \\
 u(0) &= u_0,  \quad\text{in }\quad X_0(\Om) \setminus \mathcal{M},
 \end{aligned}
 \right.
 \end{equation}
	where $V(u)$ is the pseudogradient of $E$ in Banach space $X_0(\Om),$ and $\eta(u_0)$ is the maximum of the interval of existence of $u$). As $V$ is locally Lipschitz continuous, equation (\ref{V-u}) has unique global solution $u(t,u_0).$
\begin{definition}
A subset $\mathcal{N} \subset W$ is an invariant set of descent flow of $E$ if $\{u(t,u_0): 0\leq t< \eta(u_0), u_0 \in \mathcal{N} \setminus \mathcal{M} \} \subset \mathcal{N}.$	
\end{definition}
For $0<\al\leq 1,$ we consider the space $C^{0,\al}(\bar{\Om})$ with the usual norm
$$|u|_{C^{0,\al}(\bar{\Om})}= \sup_{x \in \bar{\Om}} |u(x)|+[u]_{\al,\Om},$$
where 
$$
[u]_{\al,\Om}:=\sup_{x,y\in\Om,\,x\neq y}\frac{|u(x)-u(y)|}{|x-y|^\al}.
$$
Let us denote $C^{0,\al}_0(\bar{\Om}):=C^{0,\al}(\bar{\Om})\cap C_0(\bar{\Om}),$ where $C_0(\bar{\Om}):=\{u\in C(\Real^N)|\, u=0\mbox{ in }\Real^N\setminus \Om\}.$
The following embeddings are dense:
$C^{0,\ga}_0(\bar{\Om}) \hookrightarrow X_0(\Om) \hookrightarrow C^{0,\al}_0(\bar{\Om})  \,\ \text{for}\,\ 0<\al=s-\f{N}{p}<\ga \leq 1.$
Let us define the cone of non-negative functions:
\begin{equation}\label{P}
P:=\big\{ u \in X_0(\Om): u(x) \geq 0 \,\ \text{for all}\,\ x \in \Om \big \}.
\end{equation}
		Let $\pa P$ be the boundary of $P$ in $X_0(\Om)$. We denote
$$P_{C^{0,\ga}_0(\bar{\Om})}=\bigg\{u \in C^{0,\ga}_0(\bar{\Om}): u(x) >0 \, \text{for all}\, x \in \Om     \bigg\} . $$
We note that, $P_{C^{0,\ga}_0(\bar{\Om})}$ has non-empty interior, denoted by $P^{\circ}_{C^{0,\ga}_0(\bar{\Om})}$ and it's boundary by $\pa P_{C^{0,\ga}_0(\bar{\Om})}$.
%
%
%
Using the similar argument as in \cite{Zhang-Li}, it is easy to note that 
	$\pa P_{C^{0,\ga}_0(\bar{\Om})}$ is dense in $\pa P.$ 
				Indeed, let $u_0 \in \pa P$ and $V_{u_0}$ be an open ball of $u_0$ in $X_0(\Om)$. As $C^{0,\ga}_0(\bar{\Om})$ is dense in $X_0(\Om)$, there exists $u_1 \in C^{0,\ga}_0(\bar{\Om})$ and $P^{\circ}_{C^{0,\ga}_0(\bar{\Om})} \cap V_{u_0}$ and $u_2 \in \big(C^{0,\ga}_0(\bar{\Om}) \setminus P^\circ_{C^{0,\ga}_0(\bar{\Om})}\big)  \cap V_{u_0}.$ Then, we have,
			$tu_1+(1-t)u_2 \in C^{0,\ga}_0(\bar{\Om}) \cap V_{u_0}$ for all $t \in [0,1].$
		By continuity of the expression $tu_1+(1-t)u_2, t\in [0,1]$ w.r.t. $t$, there exists $t_0 \in [0,1]$ such that,	
			$t_0 u_1+(1-t_0)u_2 \in \pa P_{ C^{0,\ga}_0(\bar{\Om})} \cap V_{u_0}$. 

\section{Construction of pseudo-gradient vector field}\label{S-2}
In this section, we prove the following important lemma which is a key-tool to obtain our main results. To be precise, we construct the suitable pseudo-gradient field such that
$$\lim_{m \to 0^+} \frac{1}{m}\text{dist}(u+m(-V(u),P))=0,$$
for all $u \in B(u_0) \cap P, u_0 \in \pa P \setminus \mathcal{M},$ where $P$ is given in (\ref{P}).
To this motive, we state the following lemma.

\begin{lemma}\label{lem-1}
	Let $E$ be defined  in (\ref{E}). Then there exists a pseudo-gradient of $E$ such that
	the following holds:
	\begin{itemize}
		\item [(i)]
	 $P,-P$ are invariant sets of descent flow of $E$; 
	 \item [(ii)]
	  $u(t,u_0) \in {P^\circ}$ (or $-P^\circ$ ) for all $t>0$ and for all $u_0 \in P$ (or $-P$),
	  where $u(t,u_0)$ is the solution of (\ref{V-u});
	  \item [(iii)]
	  $P^\circ,-P^\circ$ remains invariant under the descent flow of $E.$
	 \end{itemize}
\end{lemma}	
\begin{proof}
	We will first construct a pseudo-gradient vector field $V$ in $X_0(\Om).$ We take
	\begin{equation}\label{B}
	B=\{ u_0 \in \pa P : E'(u_0) \neq 0, \exists  w_0 \in X_0(\Om) \,\ \text{s.t.}\,\ \norm{w_0}_{X_0(\Om)}=1\,\ \text{and}\,\ \<E'(u_0),w_0\> > \frac{4}{5}\norm{E'(u_0)} \}.
	\end{equation}
	Using the density of $C^\ga_0(\bar{\Om})$ in $X_0(\Om),$ it suffices to consider $u \in C^\ga_0(\bar{\Om}).$
	
	\textit{Suppose $u_0(x)>0$ for all $x \in \Om$ and $u_0-\la_0w_0 \in P^\circ$ for some $\la_0>0.$}\\
	Under this assumption, it is easy to note that
	$u_0-\la w_0 \in P^\circ$ for all $\la \in (0,\la_0].$
	Let us take $v_0=\frac{3}{2}\norm{E'(u_0)}w_0.$
	Then, $\norm{v_0}_{X_0(\Om)}=\frac{3}{2}\norm{E'(u_0)}<2\norm{E'(u_0)},$
	since $\norm{w_0}_{X_0(\Om)}=1$ and
	\Bea
	\<E'(u_0),v_0\>&=&\frac{3}{2}\norm{E'(u_0)}\<E'(u_0),w_0\>>\frac{1}{2}\norm{E'(u_0)}^2.\\
	\Eea
	Also, $u_0-m_0v_0 \in P^\circ$ for some $m_0>0$ sufficiently small. Infact taking $m_0=\frac{2}{3}\frac{\la_0}{\norm{E'(u_0)}}$ from hypothesis, we have 
	$u_0-\frac{2\la_0}{3\norm{E'(u_0)}}v_0=u_0-\la_0w_0 \in P^\circ.$ 
	By the continuity of $E',$ we obtain, 
	$$\norm{v_0}_{X_0(\Om)} < 2\norm{E'(u)},\,\, \<E'(u),v_0\>>\frac{1}{2}\norm{E'(u)}^2,\quad\text{for all}\quad u \in N^\circ(u_0),$$
for some open neighborhood $N^\circ(u_0) \subset X_0(\Om).$
	There exists an open neighborhood $N(u_0) \subset X_0(\Om)$ such that 
	$u-mv_0 \in P^{\circ}$for all $m \in (0,m_0]$ and for	all $u \in N(u_0) \cap P.$ 

\textit{If $u_0(x)>0$ for all $x \in \Om$ and $u_0-\la w_0 \notin P^\circ$ for any $\la>0$ (sufficiently small)},	
then there exists $x_\la \in \Om$ such that $u_0(x_\la)-\la w_0(x_\la)<0$ and  as $u_0(x)>0,$ so $\text{dist}(x_\la,\pa\Om) \to 0$ as $\la \to 0^+.$ Using
Hahn-Banach theorem, the linear functional $E'(u_0)$ can be extended to linear functional on $W^{s,p}(\Om).$ Also, it is easy to see that for $w_1=w_0-\eps$ satisfies (\ref{B}) for $\eps>0$ sufficiently small. 	Indeed, we have $\<E'(u_0),w_0\> >\frac{4}{5}\norm{E'(u_0)}$ and for $w_1=w_0 -\eps,$ we see that, for $\eps$ sufficiently small, using (\ref{B}), we get,
	\Bea
	\<E'(u_0),w_1\>&=&\int_{\R^{2N}}\frac{|u_0(x)-u_0(y)|^{p-2}(u_0(x)-u_0(y))(w_1(x)-w_1(y))}{|x-y|^{N+ps}}dxdy\\
	&-&\Iom h(u_0)w_1 dx,\\
	&=&\int_{\R^{2N}}\frac{|u_0(x)-u_0(y)|^{p-2}(u_0(x)-u_0(y))(w_0(x)-w_0(y))}{|x-y|^{N+ps}}dxdy\\
	&-&\Iom h(u_0)w_0 dx+\eps \Iom h(u_0)dx,\\
	&=&\<E'(u_0,w_0\>+\eps \Iom h(u_0)>\frac{4}{5}\norm{E'(u_0)}.
	\Eea  
	This implies, $w_1=w_0-\eps$ satisfies (\ref{B}) for sufficiently small $\eps>0.$
%
	Thus, we note that,
	\begin{equation}
	\label{eq:*1}
	\int_{\R^{2N}}\frac{|u_0(x)-u_0(y)|^{p-2}(u_0(x)-u_0(y)(w_1(x)-w_1(y)}{|x-y|^{N+ps}}dxdy-\int_{\Om \setminus \Om_\de}h(u_0)w_1dx >\frac{4}{5}\norm{E'(u_0)},
	\end{equation}
	where $\Om_\de$  is a neighborhood of $\pa\Om$ in $\Om$ for sufficiently small $\delta>0$ such that $w_1(x)<0$ for all $x \in \Om_\de.$
 We denote:
\begin{equation*}
\begin{aligned}
\R^{2N}=\Rn \times \Rn
&=\big( (\Om \setminus \Om_\de)  \cup \Om_\de \cup (\Rn \setminus \Om) \big) \times \big( (\Om \setminus \Om_\de)  \cup \Om_\de \cup (\Rn \setminus \Om) \big)
:=\bigcup_{i=1}^9 \mathcal{D}_i,
\end{aligned}
\end{equation*}
where
\begin{align*}
&\mathcal{D}_1:=(\Om \setminus \Om_\de)  \times (\Om \setminus \Om_\de),\,
\mathcal{D}_2:= (\Om \setminus \Om_\de) \times  \Om_\de,
\mathcal{D}_3:= (\Om \setminus \Om_\de) \times(\Rn \setminus \Om)  ,\\ 
&\mathcal{D}_4:= \Om_\de  \times (\Om \setminus \Om_\de),\,
\mathcal{D}_5:= \Om_\de  \times  \Om_\de,\, \mathcal{D}_6:=\Om_\de  \times (\Rn \setminus \Om),\\
&
\mathcal{D}_7:=(\Rn \setminus \Om)  \times (\Om \setminus \Om_\de),\,
\mathcal{D}_8:=(\Rn \setminus \Om)  \times  \Om_\de,\,
\mathcal{D}_9:=(\Rn \setminus \Om)  \times (\Rn \setminus \Om).
\end{align*}
Therefore,
\begin{equation*}
\begin{aligned}
&\int_{\R^{2N}}\frac{|u_0(x)-u_0(y)|^{p-2}(u_0(x)-u_0(y)) (w_1(x)-w_1(y))}{|x-y|^{N+sp}}dxdy\\
&=\sum_{i=1}^9 \int_{\mathcal{D}_i}\frac{|u_0(x)-u_0(y)|^{p-2}(u_0(x)-u_0(y)) (w_1(x)-w_1(y))}{|x-y|^{N+sp}}dxdy,
\end{aligned}
\end{equation*}
where $\mathcal{D}_i$'s are defined above. As $w_0 \in X_0(\Om), w_1 \equiv -\eps $ on $\Rn \setminus \Om,$ hence,
\begin{equation*}
\int_{\mathcal{D}_9}\frac{|u_0(x)-u_0(y)|^{p-2}(u_0(x)-u_0(y)) (w_1(x)-w_1(y))}{|x-y|^{N+sp}}dxdy=0.
\end{equation*}
Using the above observations, we have from \eqref{eq:*1} that
\begin{equation}
\begin{aligned}
&\sum_{i=1}^8 \int_{\mathcal{D}_i}\frac{|u_0(x)-u_0(y)|^{p-2}(u_0(x)-u_0(y)) (w_1(x)-w_1(y))}{|x-y|^{N+sp}}dxdy\\&-\int_{\Om \setminus \Om_\de} h(u_0)w_1 dx> \frac{4}{5}\|E'(u_0)\|.
\end{aligned}
\end{equation}
As $w_1(x)=-\eps<0$ in $\Rn \setminus \Om,$ therefore, choosing $\eps>0$ further small enough (if necessary), there exists neighborhood $\Om_\de$ of $\pa \Om$ in $\Om$ for $\de>0$ small such that, $w_1(x)<0$ in $\Om_\de$ and
\begin{equation}\label{eq:*2}
\begin{aligned}
&\int_{\mathcal{D}_1}\frac{|u_0(x)-u_0(y)|^{p-2}(u_0(x)-u_0(y)) (w_1(x)-w_1(y))}{|x-y|^{N+sp}}dxdy\\
&+\int_{\mathcal{D}_3}\frac{|u_0(x)-u_0(y)|^{p-2}(u_0(x)-u_0(y)) (w_1(x)-w_1(y))}{|x-y|^{N+sp}}dxdy\\
&+\int_{\mathcal{D}_7}\frac{|u_0(x)-u_0(y)|^{p-2}(u_0(x)-u_0(y)) (w_1(x)-w_1(y))}{|x-y|^{N+sp}}dxdy\\
&-\int_{\Om \setminus \Om_\de} h(u_0)w_1 dx> \frac{4}{5}\|E'(u_0)\|.
\end{aligned}
\end{equation}
Thus, we may define $w_2 \in X_0(\Om)$ such that
	\begin{equation*}
w_2(x)=\begin{cases}
w_1(x), \forall x \in \Om \setminus \Om_\de,\\
<0 ,\,\ \forall x \in \Om_\de,\\
0,\,\, x \in \Rn \setminus \Om,
\end{cases}
\end{equation*}
with
\allowdisplaybreaks
\begin{align}\label{eq:*3}
&\bigg|\int_{\mathcal{D}_2}\frac{|u_0(x)-u_0(y)|^{p-2}(u_0(x)-u_0(y)) (w_2(x)-w_2(y))}{|x-y|^{N+sp}}dxdy\no\\
&+\int_{\mathcal{D}_4}\frac{|u_0(x)-u_0(y)|^{p-2}(u_0(x)-u_0(y)) (w_2(x)-w_2(y))}{|x-y|^{N+sp}}dxdy\no\\
&+\int_{\mathcal{D}_5}\frac{|u_0(x)-u_0(y)|^{p-2}(u_0(x)-u_0(y)) (w_2(x)-w_2(y))}{|x-y|^{N+sp}}dxdy\no\\
&+\int_{\mathcal{D}_6}\frac{|u_0(x)-u_0(y)|^{p-2}(u_0(x)-u_0(y)) (w_2(x)-w_2(y))}{|x-y|^{N+sp}}dxdy\no\\
&+\int_{\mathcal{D}_8}\frac{|u_0(x)-u_0(y)|^{p-2}(u_0(x)-u_0(y)) (w_2(x)-w_2(y))}{|x-y|^{N+sp}}dxdy\no\\
&-\int_{ \Om_\de} h(w_2)w_2 dx\bigg| < (\frac{4}{5} -\frac{2}{3})\|E'(u_0)\|,
\end{align}
and $\|w_2\|<2$.
Therefore, (\ref{eq:*2}) and (\ref{eq:*3}) implies,
\begin{align*}
\<E'(u_0),w_2\>&=\int_{\R^{2N}}\frac{|u_0(x)-u_0(y)|^{p-2}(u_0(x)-u_0(y)) (w_2(x)-w_2(y))}{|x-y|^{N+sp}}dxdy\\&-\Iom h(u_0)w_2 dx\\
&=\sum_{i=1}^d \int_{\mathcal{D}_i}\frac{|u_0(x)-u_0(y)|^{p-2}(u_0(x)-u_0(y)) (w_2(x)-w_2(y))}{|x-y|^{N+sp}}dxdy\\&-\int_{\Om \setminus \Om_\de} h(u_0)w_2 dx-\int_{\Om_\de}h(u_0)w_2 dx\\
&=\int_{\mathcal{D}_1}\frac{|u_0(x)-u_0(y)|^{p-2}(u_0(x)-u_0(y)) (w_1(x)-w_1(y))}{|x-y|^{N+sp}}dxdy\\
&+\int_{\mathcal{D}_3}\frac{|u_0(x)-u_0(y)|^{p-2}(u_0(x)-u_0(y)) (w_1(x)-w_1(y))}{|x-y|^{N+sp}}dxdy\\
&+\int_{\mathcal{D}_7}\frac{|u_0(x)-u_0(y)|^{p-2}(u_0(x)-u_0(y)) (w_2(x)-w_2(y))}{|x-y|^{N+sp}}dxdy\\
&-\int_{\Om \setminus \Om_\de}h(u_0)w_1dx\\
&+\int_{\mathcal{D}_2}\frac{|u_0(x)-u_0(y)|^{p-2}(u_0(x)-u_0(y)) (w_2(x)-w_2(y))}{|x-y|^{N+sp}}dxdy\\
&+\int_{\mathcal{D}_4}\frac{|u_0(x)-u_0(y)|^{p-2}(u_0(x)-u_0(y)) (w_2(x)-w_2(y))}{|x-y|^{N+sp}}dxdy\\
&+\int_{\mathcal{D}_5}\frac{|u_0(x)-u_0(y)|^{p-2}(u_0(x)-u_0(y)) (w_2(x)-w_2(y))}{|x-y|^{N+sp}}dxdy\\
&+\int_{\mathcal{D}_6}\frac{|u_0(x)-u_0(y)|^{p-2}(u_0(x)-u_0(y)) (w_2(x)-w_2(y))}{|x-y|^{N+sp}}dxdy\\
&+\int_{\mathcal{D}_8}\frac{|u_0(x)-u_0(y)|^{p-2}(u_0(x)-u_0(y)) (w_2(x)-w_2(y))}{|x-y|^{N+sp}}dxdy\\
&-\int_{\Om_\de}h(u_0)w_2dx> (\frac{4}{5}-\frac{4}{5}+\frac{2}{3} ) \|E'(u_0)\|=\frac{2}{3}\|E'(u_0)\|.
\end{align*}
 We note that $u_0-m_0w_2 \in P^\circ$ for some $m_0>0$ sufficiently small. 
Considering $v_0=\frac{3}{2}\norm{E'(u_0)}\frac{w_2}{\norm{w_2}_{X_0(\Om)}},$ we note that $v_0$ satisfies:
$$\norm{v_0}_{X_0(\Om)} <2\norm{E'(u_0)},\,\, \<E'(u_0),v_0\> >\frac{1}{2}\norm{E'(u_0)}^2,\,\,
u_0-m_0v_0 \in P^\circ,$$ for some $m_0>0$ sufficiently small. By the continuity of $E'$,
 we have, $u-mv_0 \in P^\circ$ for all $m \in (0,m_0]$ and for all $u \in N(u_0) \cap P$ where $N(u_0)$ is a neighborhood of $u_0.$

Finally, \textit{suppose $u_0$ vanishes for some point $x_0 \in \Om.$}

Let $\Om_0:=\{x \in \Om:u_0(x)>0\}.$ As $u_0$ attains its interior minima at $x=x_0,$ hence $\na u_0(x_0)=0.$ There exists an open set $\Om_1$ with smooth boundary $\Om_0 \subset \Om_1 \subset \Om$ and a sequence $u_n \in C^\ga_0(\Om)$ such that
$$u_n|_{\Om_1} \in C^\ga_0(\bar{\Om}), u_n(x)>0, x \in \Om_1, u_n(x)=0, x \in \Om\setminus \Om_1 \,\ \text{for all}\,\ n=1,2,\cdots$$ and $\norm{u_n-u_0}_{X_0(\Om)} \to 0$ as $n \to \infty.$
Thus, the set $\big\{ u \in \pa P: \Om_0=\{ x \in \Om: u>0\}\,\ \text{has smooth boundary}\big\}$ is dense on $\pa P.$
So, let us assume $\Om_0$ has smooth boundary.

\textit{Claim}: We can choose $w_0$ in (\ref{B}) such that $w_0|_{\Om_0} \in X_0(\Om_0), w_0(x)=0$ for all $x \in \Om \setminus \Om_0:= \{y: u_0(y)=0 \}.$

We prove the claim by the method of contradiction. Suppose not, then for all $w \in X_0(\Om_0)$ such that with $w|_{\Om_0} \in X_0(\Om_0),$ $w(x)=0 $ for all $x \in \Om \setminus \Om_0, w$ does not satisfy (\ref{B}).
Hence, $E'(u_0)=0$ as a linear functional on the Banach space $X_0(\Om_0),$ that is, $u_0|_{\Om_0} \in X_0(\Om_0)$ 
is a positive weak solution of the equation
$$
\left\{\begin{aligned}
(-\De)^s_p u &=h(u), \quad\text{in }\quad \Om_0, \\
u &= 0, \quad \text{in}\quad \Rn \setminus \Om_0.
\end{aligned}
\right.
\leqno{(\mathcal{P}_*)}
$$
It is easy to check that $u_0|_{\Om_0}$ is a positive weak solution of $(\mathcal{P}_*).$ Using this, we can show that 
$\<E'(u_0),w\>=0$ for all $w \in X_0(\Om).$
		To see this, we calculate for $w \in X_0(\Om),$
		\Bea
		\<E'(u_0),w\>&=&\int_{\R^{2N}}\frac{|u_0(x)-u_0(y)|^{p-2}(u_0(x)-u_0(y))(w(x)-w(y))}{|x-y|^{N+sp}}dxdy\\
		&-&\int_{\R^{N}}h(u_0)w(x)dx.
		\Eea
		Using the fact that $u_0=0$ in $\Om \setminus \Om_0, w=0$ in $\Rn \setminus \Om,$ a simple computation immediately yields us,
		
		\bea\label{k-1}
		\<E'(u_0),w\>&=&\int_{\Om_0 \times \Om_0}\frac{|u_0(x)-u_0(y)|^{p-2}(u_0(x)-u_0(y))(w(x)-w(y))}{|x-y|^{N+sp}}dxdy\no\\
		&+&2\int_{\Rn \setminus \Om_0}\int_{\Om_0}\frac{|u_0(x)|^{p-2}u_0(x)(w(x)-w(y))}{|x-y|^{N+sp}}dxdy\no\\
		&-&\int_{\Om_0}h(u_0)w(x)dx.
		\eea
		We note that, $u_0$ satisfies:
		$(-\De)^s_p u_0=h(u_0).$ This evidently implies,
		\begin{equation}\label{k-2}
		\int_{\R^{2N}}\frac{|u_0(x)-u_0(y)|^{p-2}(u_0(x)-u_0(y))(w(x)-w(y))}{|x-y|^{N+sp}}dxdy=\int_{\Om_0}h(u_0)wdx.
		\end{equation}
		In a similar manner, using (\ref{k-2}), we conclude that,
		\begin{align}\label{k-3}
		&\int_{\Om_0 \times \Om_0}\frac{|u_0(x)-u_0(y)|^{p-2}(u_0(x)-u_0(y))(w(x)-w(y))}{|x-y|^{N+sp}}dxdy\no\\&+2	\int_{\Rn \setminus \Om_0}\int_{\Om_0}\frac{|u_0(x)|^{p-2}u_0(x)(w(x)-w(y))}{|x-y|^{N+sp}}dxdy\no\\
		&=\int_{\Om_0}h(u_0)w(x)dx.
		\end{align}	
		Equations (\ref{k-1}) and (\ref{k-3}) together commits,
		$\<E'(u_0),w\>=0.$ 
		For the element $w_0$ in (\ref{B}), we have
		$\<E'(u_0),w_0\>=0,$ which contradicts the condition of (\ref{B}). Hence, proved.
This evidently implies,
$\<E'(u_0),w_0\>=0,$ which contradicts the condition of (\ref{B}). This proves the claim.

Consequently, we can find $w_0 \in X_0(\Om_0)$ such that $\<E'(u_0),w_0\> >\frac{4}{5}\norm{E'(u_0)}.$ By the continuity of $E'(u_0),$ there exists $w_1 \in X_0(\Om)$ such that $w_1:=w_0-\eps \phi_1$ ($\eps$ sufficiently small), $\norm{w_1}_{X_0(\Om)} \leq 1+\eps$ satisfy (\ref{B}) and $w_1(x)<0$ for all $x \in \{y \in \Om : u_0(y)=0  \}.$ Then, we obtain for $\la>0$ sufficiently small, $u_0(x)-\la w_1(x)>0$ for all $x \in \Om.$ If $u_0-\la w_1 \notin P^\circ$ for $\la>0$ sufficiently small, we can modify $w_1$ as in the previous case.
Ergo, we can take $u_0-\la w_1 \in P^\circ$ for $\la>0$ sufficiently small. Considering $v_0=\frac{3}{2}\norm{E'(u_0)}\frac{w_1}{\norm{w_1}_{X_0(\Om)}}$, we get,
$$\norm{v_0}_{X_0(\Om)} < 2\norm{E'(u_0)},\,\, \<E'(u_0),v_0\> > \frac{1}{2}\norm{E'(u_0)}^2.$$
Using the continuity of $E'(u_0)$, it follows that 
\begin{equation}\label{v-0}
\norm{v_0}_{X_0(\Om)} <2\norm{E'(u)}\,\ \text{and}\,\ \<E'(u),v_0\> > \frac{1}{2}\norm{E'(u_0)}^2\,\ \text{for all}\,\ u \in N^\circ(u_0),
\end{equation}
where $N^\circ(u_0)$ is an open neighborhood of $u_0$ in $X_0(\Om).$
In a similar fashion, as above, we can show there exists an open neighborhood $N(u_0) \subset N^\circ(u_0) \subset X_0(\Om)$ and $m_0>0$ such that
$u-mv_0 \in P^\circ$ for all $u \in N(u_0)\cap P$ for all $m \in (0,m_0].$ By the above mentioned arguments, the family of open sets
$\big\{ N(u_0): u_0 \in \pa P \setminus \mathcal{M} \big\}$ ($u_0 \in X_0(\Om)), \{x \in \Om: u_0(x)>0  \}$ has smooth boundary) is an open covering of $\pa P \setminus \mathcal{M}.$ 
In a likewise manner, $\big\{ N(u_0): u_0 \in \pa (-P) \setminus \mathcal{M} \big\}$ is an open covering of $\pa (-P) \setminus \mathcal{M}.$ For other points $u_0 \in (X_0(\Om) \setminus \mathcal{M}) \setminus \{N(u_0): u_0 \in (\pa P \cup \pa (-P)) \setminus \mathcal{M}  \}$, $E'(u_0) \neq 0,$
there exists a neighborhood $N(u_0)$ such that 
$$\text{dist}(N(u_0), \pa P)>0,\,\, \text{dist}(N(u_0),\pa (-P))>0, $$
$$\norm{v_0}_{X_0(\Om)} < 2\norm{E'(u)},\,\, \<E'(u),v_0\>>\frac{1}{2}\norm{E'(u)}^2,$$
for all $u \in N(u_0).$ 
Therefore, the family of open sets $\Theta=\{N(u_0): u_0 \in X_0(\Om) \setminus \mathcal{M}  \}$ is an open covering of the paracompact space $X_0(\Om) \setminus \mathcal{M}.$ Hence, $\Theta$ has a locally finite refinement $\{U_{u^{-i}_0}\}_{i \in I},$ that is, each $U_{u^{-i}_0}\}$ is one of $N(u_0)$ and each $u \in X_0(\Om)$ has a neighborhood $B(u)$ such that $B(u) \subset U_{u^{-i}_0} \neq \emptyset$ only for finitely many $i \in I.$
Let us define for each $i \in I,$ the distance function 
$$\rho_i(u)=\text{dist}(u,X_0(\Om) \setminus U_{u^{i}_0}),$$ and for $u \in X_0(\Om)$, the vector field 
$$V(u)=\sum_{i \in I}\frac{\rho_i(u)}{\sum_{j \in I} \rho_j(u)}v_i,\,\ 
\text{where} \,\ v_i=u^i_0 \,\ \text{and}\,\ v_0=u_0.$$ 
Note that, since $\{U_{u^{i}_0}\}_{i \in I}$ is locally finite, sums in the above expressions defining $V$ are finite. Therefore, $V$ is locally Lipschitz. As $\rho_i$
vanishes outside $U_{u^{i}_0},$ $V(u)$ is a convex combination of finite elements satisfying (\ref{v-0}). This implies, $V(u)$ ia a pseudo-gradient of $E.$
We notice that $u-mV(u) \in P^\circ$ for all $u \in B(u_0) \cap P$ and $m>0$ sufficiently small.
%
%
Hence, we have,
\begin{equation}\label{**}
\lim_{m \to 0^+} \frac{1}{m} \text{dist}(u+m(-V(u)),P)=0.
\end{equation}
Consequently, (\ref{**}) is true for $u_0 \in P^\circ \setminus \mathcal{M}.$
By Corollary 2.3.1 of \cite{Guo-Sun} and the semigroup property, it is easy to see that $P,-P$ are all invariant subsets of descent flow of $E.$ This proves (i).

For $u_0 \in \pa P \setminus \mathcal{M},$ the solution of (\ref{V-u}), $u(t,v) \in P^\circ$ with initial condition $v$ for $t>0$ sufficiently small  and for all $v \in B(u_0) \cap P.$ This proves (ii).

Hence,  we conclude that $P^\circ$ is an invariant set of descent flow of $E.$ Similarly, $-P^\circ$ is also an invariant set of descent flow of $E.$ This completes the proof.

\end{proof}

\section{Palais-Smale Condition}\label{S-3}
In this section, we show that the energy functional $E$ satisfies the Palais-Smale (PS) condition.
We recall that 
 $\{u_n\}$ is a Palais-Smale sequence (in short, PS sequence)  of $E$ at level $c$ if $E(u_n)\to c$ and $E'(u_n)\to 0$ in $(X_0(\Om))'$, the dual space of $X_0(\Om)$. Moreover,  we say that 
	$E$ satisfies (PS)$_c$ condition if $\{u_n\}$ is any (PS) sequence in $X_0(\Om)$ at level $c$ implies $\{u_n\}$ has a convergent subsequence in $X_0(\Om)$.

\begin{lemma}\label{lem-2}
	Let $a_1, d_1< \la_1$, $E$ be defined as in (\ref{E})
 and $h$ satisfy  the assumptions $(A_1), (A_2), (A_3)$  Then $E$ satisfies the (PS)-condition on $X_0(\Om).$
\end{lemma}	
\begin{proof}
Let $\{u_n\}_{n \geq 1}$ be a sequence such that $|E(u_n)| \leq C$ and $\norm{E'(u_n)}_{X'_0(\Om)} \to 0.$ By the assumption $(A_2),$ 
we note that,
$$H(u) \leq \frac{(a_1+\eps)}{p}u^p+C_1u \,\ \text{for all}\,\ u \geq 0,\,\ $$ and
$$H(u) \leq \frac{(d_1+\eps)}{p}|u|^p-C_1u \,\ \text{for all}\,\ u \leq 0,$$
where $C_1=\min_{u \in [0,-t_0]}h(u)$ for all $u \leq 0.$
This implies,
\bea\label{C-2}
C &\geq& E(u_n)=\frac{1}{p}\norm{u_n}_{X_0(\Om)}^p-\Iom H(u_n)dx,\no\\
&\geq& \frac{1}{p}\frac{(p-1)}{\la_1}\eps \norm{u_n}_{X_0(\Om)}^p-C_2\norm{u_n}_{X_0(\Om)}.
\eea
Hence, $\norm{u_n}_{X_0(\Om)}$ is bounded.
Let $$F(u):=\frac{1}{p}\norm{u}_{X_0(\Om)}^p, G(u):=\Iom H(u)dx.$$
So,  for $p \geq 2,$ we have,
$$\<F'(u_n)-F'(u_m),u_n-u_m\> \geq C \norm{u_n-u_m}^p_{X_0(\Om)}.$$
This yields us,
\begin{equation}\label{C-3}
\norm{F'(u_n)-F'(u_m)}_{X'_0(\Om)} \geq C \norm{u_n-u_m}_{X_0(\Om)}^{p-1}.
\end{equation}
For $1<p<2$ we obtain,
\begin{equation}\label{C-4}
\norm{u_n-u_m}_{X_0(\Om)} \leq C \norm{F'(u_n)-F'(u_m)}_{X'_0(\Om)}.
\end{equation}
We also have,
\begin{equation}\label{C-5}
\norm{E'(u_n)-E'(u_m)}_{X'_0(\Om)} \geq \norm{F'(u_n)-F'(u_m)}_{X'_0(\Om)}-
\norm{{G}'(u_n)-{G}'(u_m)}_{X'_0(\Om)}.
\end{equation}
Also, as ${G}'$ is compact from $X_0(\Om) $ to $X'_0(\Om),$ ${G}'(u_n)$ has a convergent subsequence. Using (\ref{C-3})-(\ref{C-5}) we conclude that $\{u_n\}$ has a convergent subsequence. Hence, $E$ satisfies PS-condition.

\end{proof}	

\begin{lemma}\label{lem-3}
	Let $(a_1,d_1) \notin \Sigma_p$ and
	the assumptions $(A_1), (A_2)$ and $(A_3)$ hold. Then $E$ satisfies (PS) condition on $X_0(\Om).$
\end{lemma}	
\begin{proof}
	Let $\{u_n\}_{n \geq 1}$ be a sequence such that $|E(u_n)| \leq C$ and $\norm{E'(u_n)}_{X'
		_0(\Om)} \to 0$ as $n \to \infty.$
\\
\textit{Claim}: $\{u_n\}$ is bounded in $X_0(\Om).$
\\
To see the proof, let us take $g(u)=h(u)-a_1(u^+)^{p-1}+d_1(u^-)^{p-1}.$
Then, $$\lim_{|u| \to +\infty}\frac{g(u)}{|u|^{p-2}u}=0.$$	
	We write $(\mathcal{P})$ as:
	\begin{equation*}
		\left\{\begin{aligned}
			(-\De)^s_p u  &=a_1(u^+)^{p-1}-d_1(u^-)^{p-1}+g(u)\quad\text{in }\quad \Om, \\
			u &= 0  \quad\text{in }\quad \Rn \setminus \Om.
		\end{aligned}
		\right.
	\end{equation*}
Hence, for all $h \in X_0(\Om)$	we have,
\bea\label{C-6}
\<E'(u_n),h\>&=&\int_{\R^{2N}}\frac{|u_n(x)-u_n(y)|^{p-2}(u_n(x)-u_n(y))(h(x)-h(y))}{|x-y|^{N+sp}}dxdy\no\\
&-&\Iom a_1(u_n^+)^{p-1}h+\Iom d_1(u_n^-)^{p-1}h-\Iom g(u_n)h \to 0,
\eea
as $n \to \infty.$	
Let us define $A: X_0(\Om) \to X'_0(\Om)$ by
\bea\label{C-7}
\<A(u),v\>&=&\int_{\R^{2N}}\frac{|u(x)-u(y)|^{p-2}(u(x)-u(y))(v(x)-v(y))}{|x-y|^{N+sp}}dxdy\no\\
&-&\Iom a_1(u^+)^{p-1}v+\Iom d_1(u^-)^{p-1}v \,\ \text{for all}\,\ v \in X_0(\Om).
\eea
\\
\textit{Claim}: There exists $C_0>0$ such that
\begin{equation}\label{A-bound}
 \norm{A\bigg(\frac{u}{\norm{u}_{X_0(\Om)}}\bigg)}_{X'_0(\Om)} \geq C_0>0 \,\ \text{for all}\,\ u \in X_0(\Om), u \neq 0.
 \end{equation}
We prove the claim by the method of contradiction. Suppose not, then there exists a sequence $v_n$ and $\eps_n \to 0$ such that
$$\bigg|\<A \big( \frac{v_n}{\norm{v_n}_{X_0(\Om)} } \big) ,w \>\bigg|
	 \leq \eps_n \norm{w}_{X_0(\Om)} \,\ \text{for all}\,\ w \in X_0(\Om).$$
Let $h_n=\frac{v_n}{\norm{v_n}_{X_0(\Om)}}.$ From (\ref{C-7}) we obtain,
$$|\<A(h_n),w\>| \leq \eps_n \norm{w}_{X_0(\Om)},$$ which implies,
\bea\label{C-8}
|\int_{\R^{2N}} \frac{|h_n(x)-h_n(y)|^{p-2}(h_n(x)-h_n(y))(w(x)-w(y))}{|x-y|^{N+sp}}dxdy
&-&\Iom \bar{a_1} (h_n^+)^{p-1}w \no\\
&+& d_1 (h_n^-)^{p-1}w| \leq \eps_n \norm{w}_{X_0(\Om)}\no\\
\eea
for all $w \in X_0(\Om).$
As $(h_n)$ is bounded in $X_0(\Om),$ then there exists $v_0 \in X_0(\Om)$ such that  up to a subsequence, $h_n \deb v_0$ and $h_n \to v_0$ in $L^p(\Om),$ $h_n \to v_0$ a.e. in $\Rn$.
Taking $w=w_n=h_n-v_0$ in (\ref{C-8}) we have,
\begin{align*}
&\limsup_{n \to \infty}
|\int_{\R^{2N}} \frac{|h_n(x)-h_n(y)|^{p-2}(h_n(x)-h_n(y))(h_n(x)-v_0(x)-h_n(y)+v_0(y))}{|x-y|^{N+sp}}dxdy\\
&-\Iom {a_1} (h_n^+)^{p-1}(h_n-v_0) + \Iom d_1 (h_n^-)^{p-1}(h_n-v_0)| =0,
\end{align*}
since $\norm{w_n}_{X_0(\Om)}=\norm{w}_{X_0(\Om)}$ is bounded and $\eps_n \to 0$ in (\ref{C-8}). Hence,  it follows that $h_n \to v_0$ in $X_0(\Om).$ Therefore, in (\ref{C-8}) we have letting $n \to \infty,$
\Bea
\int_{\R^{2N}} \frac{|v_0(x)-v_0(y)|^{p-2}(v_0(x)-v_0(y))(w(x)-w(y))}{|x-y|^{N+sp}}dxdy
&-&\Iom {a_1} (v_0^+)^{p-1}w\\
& +& \Iom d_1 (v_0^-)^{p-1}w =0,
\Eea
for all $w \in X_0(\Om).$
This yields us that: $(a_1,d_1) \in \Sigma_p^s(\Om)$ which is a contradiction. This proves the claim.
For $h=u_n$ in (\ref{C-6}), we have,
$$\bigg|\<E'(u_n),\frac{u_n}{\norm{u_n}_{X_0(\Om)}}\>\bigg| \leq \norm{E'(u_n)}_{X'_0(\Om)} \to 0 \,\ \text{as}\,\ n \to \infty,$$
so there exists $\eps_0>0$ such that for $n$ large enough, we have,
\begin{equation}\label{C-9}
|\norm{u_n}_{X_0(\Om)}^p-\Iom a_1 (u_n^+)^p+\Iom d_1 (u_n^-)^p-\Iom g(u_n)u_n|
\leq \eps_0 \norm{u_n}_{X_0(\Om)}.
\end{equation}
This together with  assumption $(A_2)$ gives us, 
\begin{equation}
\norm{u_n}_{X_0(\Om)}^p \leq C|u_n|_p + \eps_0 \norm{u_n}_{X_0(\Om)},
\end{equation}
for  some $C>0,$$n$ large enough. So, we have,
\Bea
C_0 \la_1^{\frac{p-1}{p}}|u_n|^{p-1}_p 
&\leq& C+\eps |u_n|_p^{p-1}. 
\Eea
Using a standard argument, it follows that $u_n$ is bounded in $X_0(\Om).$
 Now, proceeding as in Lemma \ref{lem-2}, we obtain a convergent subsequence of $\{u_n\}.$ 
This completes the proof.
	\end{proof}

\section{Construction of a curve connecting the interior of  positive and negative cones}\label{S-4}
In this section, we construct a curve $\Ga_1$ in $X_0(\Om)$ connecting the interior of the positive and negative cones and the energy functional assumes negative values on the curve $\Ga_1.$
\begin{lemma}\label{lem-4}
	Let $(a,d) \in S$ where $S$ is defined in (\ref{S}). Then there exists a path $\Ga_1 \subset X_0$ connecting $P^\circ$ and $-P^\circ$ such that
	$$E_{(a,d)}(u)<0 \,\ \text{for all}\,\ u \in \Ga_1.$$
\end{lemma}	
\begin{proof}
	Using the fact that $(a,d) \in S,$  there exists $(a',d') \in \Ga$ ($\Ga$ is defined in (\ref{Ga})) such that $a' \in (\la_1,a), d' \in (\la_1,d)$ and $\phi_0 \in X_0(\Om)$ such that
		\begin{equation*}
		\left\{\begin{aligned}
			(-\De)^s_p \phi_0  &=a'(\phi_0^+)^{p-1}-d'(\phi_0^-)^{p-1}\quad\text{in }\quad \Om, \\
			\phi_0 &= 0  \quad\text{in }\quad \Rn \setminus \Om.
		\end{aligned}
		\right.
	\end{equation*}
We note that $\phi_0$ changes sign. 
It is easy to see that
$$(\phi_0(x)-\phi_0(y))(\phi_0^+(x)-\phi_0^+(y)) \geq (\phi_0^+(x)-\phi_0^+(y))^2.$$
Proceeding as before we can show,
\begin{equation}\label{a'}
a' \Iom (\phi_0^+)^p \geq \int_{\R^{2N}}\frac{|\phi_0^+(x)-\phi_0^+(y)|^{p}}{|x-y|^{N+sp}}.
\end{equation}
Similarly, we have,
\begin{equation}\label{d'}
d' \Iom (\phi_0^-)^p \geq \int_{\R^{2N}}\frac{|\phi_0^-(x)-\phi_0^-(y)|^{p}}{|x-y|^{N+sp}}.
\end{equation}
Using a simple calculation, we have, for all $t \in [0,1],$
\begin{align*}
	&E_{(a,d)}(t\phi_0^++(1-t)\phi_0^-)\\
	&\leq \frac{t^p}{p}\int_{\R^{2N}}\frac{|\phi_0^+(x)-\phi_0^+(y)|^p}{|x-y|^{N+sp}}dxdy+\frac{(1-t)^p}{p}\int_{\R^{2N}}\frac{|\phi_0^-(x)-\phi_0^-(y)|^p}{|x-y|^{N+sp}}dxdy\\
	&-\frac{a}{p}\Iom (t\phi_0^+)^p-\frac{d}{p}\Iom ((1-t)\phi_0^-)^p\\
\end{align*}	
This togetherwith (\ref{a'}) and (\ref{d'})  (as $a>a',d>d')$ implies,
$$E_{(a,d)}(t\phi_0^++(1-t)\phi_0^-) \leq \frac{(a'-a)}{p}\Iom (t\phi_0^+)^p+\frac{(d'-d)}{p}\Iom ((1-t)\phi_0^-)^p<0.$$
As $\{t\phi_0^++(1-t)\phi_0^-: 0\leq t \leq 1 \}$ is a compact subset of $X_0(\Om),$ we can choose a curve $\Ga_1$ which is defined by:
$$\Ga_1:=\{ \Ga_1(t) : 0 \leq t \leq 1  \}$$
such that $\Ga_1(t)$ is very close to $\{t\phi_0^++(1-t)\phi^- \}$ for each $t \in [0,1]$ and $\Ga_1(0) \in P^\circ, \Ga_1(1) \in -P^\circ.$ Also, we have,
$E_{(a,d)}(\Ga_1(t))<0$ for all $t \in [0,1].$
This completes the proof.
\end{proof}	

\section{Proof of  the Main Results}\label{S-5}
In this section, we prove our main  existence results, Theorem \ref{Theorem-1} and Theorem \ref{Theorem-2}.
First, we will prove the existence of distinct invariant sets on which the energy functional corresponding to $(\mathcal{P})$ is bounded from below. One can observe that the (PS)-condition yields the solution of (\ref{V-u}) tends to a critical point of $E$ in $X_0(\Om)$ or goes to a negative energy. Using this, we establish the existence of an invariant set in the complement of $P \cup (-P).$ We also prove that there exists a 
connected component $\Theta$ of $\pa U_1$ such that $\Theta \cap P^\circ \neq \emptyset, \Theta \cap (-P^\circ) \neq \emptyset, \Theta \cap [X_0(\Om) \setminus (-P \cup P)] \neq \emptyset.$

\subsection{Proof of Theorem \ref{Theorem-1}}

 From lemma \ref{lem-2}, we note that $E$ satisfies PS condition. In a likewise manner  as in (\ref{C-2}), one can check that $E$ is bounded from below in $X_0(\Om).$
As $a_0,d_0 \in S,$ using  the argument of Lemma \ref{lem-4}, it is easy to show the existence of a bounded path $\Ga_1$ connecting $P^\circ,-P^\circ$ on which the energy $E_{(a_0,d_0)}(\cdot)$ assumes negative values on that path. Now, using the compactness
of the path $\Ga_1$ in $X_0(\Om),$ we can see that 
$E_{(a,d)}(u)<-\de \,\ \text{for all}\,\ u \in \Ga_1$ and for some $\de>0.$
This implies, $$E_{(a_0,d_0)}(tu) \leq -t^p \de \,\ \text{for all}\,\ u \in \Ga_1, t>0.$$
Choosing $\eps$ small enough and using the assumption $(A_1)$ we obtain,
\begin{equation}\label{+1}
|h(u)-a_0|u|^{p-2}u|< \eps |u|^{p-2}u \,\ \text{for}\,\ u>0\,\
\text{small, that is,}\,\ (u \to 0^+),
\end{equation}
and
\begin{equation}\label{+2}
|h(u)-d_0|u|^{p-2}u|< \eps |u|^{p-2}u \,\ \text{for}\,\ -u>0\,\
\text{small, that is,}\,\ (u \to 0^-).
\end{equation}
Hence, for $t$ sufficiently small, using (\ref{+1}) and (\ref{+2}), we get,
\bea\label{D-1}
E(tu)&=&E_{(a_0,d_0)}(tu)-\Iom \bigg[ H(u)-\frac{a_0}{p}(tu^+)^p-\frac{d_0}{p}(tu^-)^p \bigg]\no\\
&\leq& -t^p \de +\Iom \frac{\eps t^p|u|^p}{p}<0 \,\ \text{for all}\,\ u \in \Ga_1.
\eea
This evidently yields the existence of a path $\Ga_2=t\Ga_1 (t>0$ small enough) connecting $P^\circ,-P^\circ$ such that $E$ assumes negative values on $\Ga_2.$
Let us construct the following set:
$$V_1=\big\{ h \in X_0(\Om): \,\ u(t_0,h) \in P^\circ \,\ \text{for some}\,\ t_0 \geq 0 \big\}.$$
By the regularity result, any weak solution of $(\mathcal{P})$, 
	$u \in C_0^{0,\ga}(\bar{\Om})$ for some $\ga\in (\al,1],$ (see \cite[Theorem 1]{L}) and using Strong Maximum Principle,
{[\cite{Quaas-17}, Theorem 1.2]},
 we obtain, $\mathcal{M} \cap (\pa P \cup (-\pa P))=\emptyset.$ Thus, for any positive weak solution $u_0 \in \pa P,$ there exists $\{u_n\} \subset \pa P_{X_0(\Om)}, n=1,2,3,\cdots$ such that $\norm{u_n-u_0}_{X_0(\Om)} \to 0$ as $n \to \infty.$ Therefore, $E'(u_n-u_0) \to 0$ as $n \to \infty.$
As $E \in C^1(X_0(\Om),\R),$ it is easy to get $\norm{u_n-u_0}_{C_0^\ga(\bar{\Om})} \to 0.$
We have already mentioned that 
		$\pa P_{C^\ga_0(\bar{\Om})}$ is dense in $\pa P$ for $\al<\ga\leq1.$ 
		So,	for any positive weak solution $u_0 \in \pa P,$ there exists $\{u_n\}_{n \geq 1} \subset \pa P_{C^\ga_0(\bar{\Om})}$ such that $\norm{u_n-u_0}_{X_0(\Om)} \to 0$ as $n \to \infty.$ As $E \in C^1(X_0(\Om),\R),$
		therefore, $E'(u_n-u_0) \to 0$ as $n \to \infty$ in $(X_0(\Om))^* $ ( where $(X_0(\Om))^*$ is the dual of $X_0(\Om)$). As $X_0(\Om)\hookrightarrow C^\al_0(\bar{\Om}),$ therefore,  $\big((-\De)^s_p\big)^{-1}:L^{\infty}(\Om) \to C^\ga_0(\bar{\Om})$ is compact, continuous, one-to-one operator, order-preserving operator for some $\ga\in(\al,1],$ and hence, $\norm{u_n-u_0}_{C_0^\ga(\bar{\Om)})} \to 0$ as $n \to \infty$.

Hence, $u_0 \in \pa P_{C^\ga_0(\bar{\Om})},$ which contradicts the Strong Maximum Principle. So, by Lemma \ref{lem-1}, $P^\circ,-P^\circ$ both are invariant sets of descent flow of $E$ in $X_0(\Om).$ By the continuous dependence of solution of (\ref{V-u}), it is easy to see the following:
\begin{itemize}
	\item [(i)]
	$V_1$ is an open invariant set of descent flow of $E$ in $X_0(\Om).$
	\item [(ii)]
	$\pa V_1$ is an invariant set of descent flow of $E$ in $X_0(\Om).$
	\item [(iii)]
	$V_1, \pa V_1$ lies in the complement of $P \cup (-P).$
	\end{itemize}
Therefore, by (\ref{D-1}) we conclude that:
$$\inf_{u \in P^\circ}E(u) \in (-\infty,0),\inf_{u \in -P^\circ}E(u) \in (-\infty,0),\inf_{u \in \pa V_1}E(u) \in (-\infty,0).$$
Assuming (without loss of generality) that $E$ has atmost finite number of critical points on any invariant set of descent flow of $E,$ we obtain that any solution of (\ref{V-u}) goes to negative energy or approaches a critical point.
Thus, there exists   $u_1 \in P^\circ, u_2 \in -P^\circ, u_3 \in \pa V_1$ such that $E'(u_i)=0$ for $i=1,2,3$ and $E$ attains infimum on $u_1, u_2, u_3$ on the corresponding sets $P^\circ, -P^\circ$ and $\pa V_1$ respectively.
This proves our result.

\subsection{Proof of Theorem \ref{Theorem-2}}
%
By Lemma \ref{lem-3}, $E$ satisfies the PS condition.
 Using assumption $(A_1)$ and $h(0)=0, a_0,d_0<\la_1;$ for $0<\eps_0<\la_1,$
there exists $\de>0$ such that $|h(t)|<(\la_1-\eps_0)|t|^{p-1}$ for all $|t|<\de.$ Using the definition of $E,$ it is not difficult to  see  that 0 is a strict local minima of $E.$	\\
\textit{Claim}: There exists a neighborhood of 0 in $X_0(\Om)$ which is an invariant descent flow of $E.$\\
	To see the proof, let us consider the neighborhoods
	$U_{\frac{1}{n}}=\{u \in U_{\de_0}: E(u)<\frac{1}{n} \}$ for all $n \in \N.$
		Now, for $n_0$ large enough and
	for $u_0 \in U_{\frac{1}{n_0}}, u_0 \neq 0,$ we have,
	$$\frac{1}{p}\eps_0 \frac{1}{\la_1}\norm{u(t,u_0)}_{X_0(\Om)}^p < E(u(t,u_0)) \leq E(u_0)<\frac{1}{n_0}.$$
	Hence, we conclude that $U_{\frac{1}{n_0}}$ is an open, invariant set of descent flow of $E$ in $X_0(\Om).$ This proves the claim.
	
	Let us denote: $U_0=U_{\frac{1}{n_0}}$ and
	 $$U_1=\{u_1 \in X_0(\Om): u(t',u_1) \in U_0 \,\ \text{for some}\,\ t'>0  \}.$$
It is easy to check that
	$U_1$ is an open invariant set of descent flow of $E$ in $X_0(\Om).$
Now, using the continuous dependence of $u(t,u_1)$ on the initial value $u_1,$ one can check that $\pa U_1$ is an invariant set of descent flow of $E$ if $\pa U_1 \neq \emptyset.$
We note that $E$ is bounded below by some positive constant on $\pa U_1.$
Hence, $E(u(t,u_0)) \geq C>0$ for all $u_0 \in \pa U_1$ and for some $C>0.$
Now, let us define the following sets:
\begin{equation}\label{V-1}
V_1:=\big\{h \in X_0(\Om): u(t_0,h) \in P^\circ \,\ \text{for some}\,\ t_0 \geq 0 \big\}
\end{equation}
	and
	\begin{equation}\label{V-2}
	V_2:=\big\{h \in X_0(\Om): u(t_0,h) \in -P^\circ \,\ \text{for some}\,\ t_0 \geq 0 \big\}.
	\end{equation}
	By Lemma \ref{lem-1}, 
	we conclude that both $V_1$ and $V_2$ are open invariant sets of descent flows of $E$ in $X_0(\Om).$
 This yields us: $\mathcal{M}$	and $\pa P \cup (-\pa P)$ are disjoint (as in the proof of Theorem \ref{Theorem-1}) and by Lemma \ref{lem-1}, we obtain $\pa V_i \cap P=\emptyset, \pa V_i \cap (-P)=\emptyset$ for $i=1,2.$	

\textit{Claim}: $\pa U_1 \neq \emptyset$ and there exists at least one connected component $\Theta \subset \pa U_1$ such that
$$\Theta \cap P^\circ \neq \emptyset, \Theta \cap (-P^\circ) \neq \emptyset, \Theta \cap (X_0(\Om)\setminus (-P\cup P)) \neq \emptyset.$$

\textit{Proof of Claim}: As $a_1,d_1 \in S,$ using Lemma \ref{lem-4}, we can choose a path $\Ga$ with no self-intersections on $\pa B_1$ (where $B_1$ is the unit ball in $X_0(\Om))$ 
connecting $P^\circ$ and $-P^\circ$ and
$E_{(a,d)}(u) <-\de<0$ for all $u \in \Ga_1$ for some $\de>0.$
For details, see Theorem 4.1 of \cite [page-27]{Whyburn}.
	Using $(A_2)$ we have for $\de>0,$
	$$\frac{\de}{p\la_1}\text{max}\{\norm{u}_{X_0(\Om)}^p: u\in \Ga \} < \text{min}\{ 
	|E_{(a,d)}(u): u \in \Ga| \}.$$
Hence, for all $u \in X_0(\Om)$ we have,
\bea\label{D-3}
\Iom |H(u^+)-\frac{a_1}{p}(u^+)^p|+\Iom |H(u^-)-\frac{d_1}{p}(u^-)^p|
&\leq&\frac{\de}{p}\frac{1}{\la_1}\norm{u}_{X_0(\Om)}^p+C\norm{u}_{X_0(\Om)}.
\eea
Therefore, for all $u \in \Ga, t\geq 0,$ we get,
\Bea
E(tu)
&\geq&t^p\max \{E_{(a_1,d_1)}(u): u \in \Ga\}+\frac{\de}{p}\frac{1}{\la_1}t^p \max\{\norm{u}_{X_0(\Om)}^p: u\in \Ga  \}\\
&+&tC\max \{\norm{u}_{X_0(\Om)}: u\in \Ga\}.
\Eea
This implies, as $t \to +\infty, E(tu) \to +\infty$ uniformly for $u \in \Ga$ with $\norm{u}_{X_0(\Om)}=1.$

Let us consider the surface defined by:
$Z=\{tu : t \geq 0, u \in \Ga  \}$ which is isomorphic to the closed upper half-plane in $\R^2.$ 
Using Lemma 4 of \cite{Dancer-1}, there exists at least one connected component $\Theta \subset (\pa U_1 \cap Z)$ such that $\Theta \cap P^\circ \neq \emptyset, \Theta \cap (-P^\circ)\neq \emptyset, \Theta \cap (X_0(\Om) \setminus (-P \cup P)) \neq \emptyset.$
 By Lemma \ref{lem-1}, using the invariance of the sets $P \cap \pa U_1, (-P) \cap \pa U_1$
 under the decent flow of $E$ and connectedness of $\pa U_1,$ we conlude that every solution of (\ref{V-u})
 goes to negative energy or tends to a critical point in $X_0(\Om).$ This enables the existence of at least three solutions with the desired properties.
Therefore, we have,
$$u(t,u_0) \to u_1 \in M\cap P \,\ \text{in}\,\ X_0(\Om) \,\ \text{as}\,\ t \to \infty, \,\ \text{for all}\,\ u_0 \in P\cap \pa U_1,$$
$$u(t,u_0) \to u_2 \in M\cap (-P) \,\ \text{in}\,\ X_0(\Om) \,\ \text{as}\,\ t \to \infty, \,\ \text{for all}\,\ u_0 \in (-P)\cap \pa U_1,$$
and 
$$u(t,u_0) \to u_3 \in M \,\ \text{in}\,\ X_0(\Om) \,\ \text{as}\,\ t \to \infty \,\ \text{for all}\,\ u \in \pa U_1 \setminus (V_1 \cup V_2).$$
Hence, we obtain that $(\mathcal{P})$ has at least three solutions 
$u_1 \in P\cap \pa U_1, u_2 \in (-P) \cap \pa U_1, u_3 \in \pa U_1 \setminus (V_1 \cup V_2).$ We have that,
 $u_1 \in P^\circ, u_2 \in -P^\circ, u_3$ is sign-changing and there exists $C>0$ such that $E(u_i) \geq C$ for all $i=1,2,3.$ This completes the proof.

\textbf{Acknowledgement}:
The author wishes to thank Prof. Mousomi Bhakta for reading a previous version of this paper and giving helpful suggestions and comments. This research is supported by the Austrian Science Foundations (FWF), Project number  P28010 and  the Czech Science Foundation, project GJ19–14413Y.

\bibliography{biblio}
\bibliographystyle{acm}

%
%
%
%
%
%
%
%
%
\end{document}